\documentclass[11pt]{article}

\usepackage{hyperref}
\hypersetup{colorlinks=true,linkcolor=blue,urlcolor=black,citecolor=blue,breaklinks}
\usepackage{math}
\usepackage{fullpage}
\usepackage[T1]{fontenc}
\usepackage{color}
\usepackage{enumerate}
\usepackage{anyfontsize}
\usepackage{cite}

\begin{document}

\title{On the Approximation of Toeplitz Operators for \\Nonparametric $\Hinf$-norm Estimation}

\author{Stephen Tu, Ross Boczar, Benjamin Recht}
\date{University of California, Berkeley}
\maketitle


\begin{abstract}
Given a stable SISO LTI system $G$, we investigate the problem of estimating the $\Hinf$-norm of $G$, denoted $\norm{G}_\infty$, when $G$ is only accessible via noisy observations. Wahlberg et al.~\cite{wahlberg2010non} recently proposed a nonparametric algorithm based on the power method for estimating the top eigenvalue of a matrix. In particular, by applying a clever time-reversal trick, Wahlberg et al. implement the power method on the top left $n \times n$ corner $T_n$ of the Toeplitz (convolution) operator associated with $G$. In this paper, we prove sharp non-asymptotic bounds on the necessary length $n$ needed so that $\norm{T_n}$ is an $\varepsilon$-additive approximation of $\norm{G}_\infty$. Furthermore, in the process of demonstrating the sharpness of our bounds, we construct a simple family of finite impulse response (FIR) filters where the number of timesteps needed for the power method is arbitrarily \emph{worse} than the number of timesteps needed for parametric FIR identification via least-squares to achieve the same $\varepsilon$-additive approximation.
\end{abstract}


\section{Introduction}

Given a stable discrete-time, linear time-invariant (LTI) system $G$, it is often desirable to compute an upper bound on the $\Hinf$-norm of $G$, denoted $\norm{G}_\infty$. For instance, the small-gain theorem \cite[Section 5.4]{khalil2002nonlinear} states that knowledge of an upper bound $\gamma$ immediately implies the feedback interconnection of $G$ with any system $\Delta$ satisfying $\norm{\Delta}_\infty \leq 1/\gamma$ is stable.  Furthermore, knowledge of $\Hinf$-norm bounds can be incorporated into the design of robust controllers~\cite{glover1989robust}.

When the model corresponding to $G$ is known, either via its state-space or transfer function representation, the exact computation of its $\Hinf$-norm is a well-studied problem, with algorithms that work efficiently in either representation~\cite{bruinsma1990fast}. However, when $G$ is only known from input/output observations, a different approach is needed.

Recently, Wahlberg et al. \cite{wahlberg2010non} proposed an algorithm based on the power method for estimating the top eigenvalue of a matrix.  The advantage of this algorithm is that no system identification is required. Letting $T(g)$ denote the Toeplitz (convolution) operator associated with the system $G$, the key idea is to run power iteration on the $n \times n$ upper-left submatrix of $T(g)$, which we denote as $T_n(g)$. Their main insight was showing how to compute the adjoint matrix-vector product $T_n(g)^* u$ for any input $u$, by a clever time-reversal trick. This makes it possible to implement power iteration on the matrix $T_n(g)^* T_n(g)$ without explicit knowledge of $T_n(g)$. 

Since $\norm{T_n(g)} \rightarrow \norm{T(g)}$ as $n \rightarrow \infty$, one reasonably expects the algorithm to converge to the desired $\Hinf$-norm of $G$ as the system runs for more timesteps. Wahlberg et al. observe that this occurs empirically, but do not provide a finite-time theoretical analysis of the necessary length $n$.

This paper establishes a finite condition on $n$ such that $\norm{G}_\infty \leq \norm{T_n(g)} + \varepsilon$ holds. We primarily leverage the body of work on the convergence properties of finite sub-sections of Toeplitz operators by B{\"{o}}ttcher and Grudsky \cite{bottcher2000book}. For a given system $G$, with stability radius $\rho \in (0, 1)$, we prove roughly that as long as \begin{align*} n \geq &\; \Omega\left( \frac{1}{(1-\rho)^2} \sqrt{\frac{\norm{G}_\infty}{\varepsilon}} \right)\;, \end{align*} then the additive error of $\norm{T_n(g)}$ is bounded by $\varepsilon$ as desired. A simple two degree FIR filter shows that our bound is essentially sharp.

We conclude by remarking on the surprising observation revealed by our analysis that the number of samples (timesteps) needed to estimate the $\Hinf$-norm with Wahlberg et al.'s nonparametric method can be arbitrarily larger than the number of samples needed to carry out FIR system identification, as analyzed in recent work by Tu et al.~\cite{tuboczar}. This raises an interesting question of whether or not the gap between $\Hinf$-norm estimation and FIR approximation in the $\Hinf$-norm is purely algorithmic or information-theoretic.

\section{Related Work}

The main inspiration for this work is the nonparametric estimation procedure of Wahlberg et al.~\cite{wahlberg2010non}, who use a time-reversal scheme to query the adjoint of an unknown LTI system.  An asymptotic analysis of this algorithm is provided by Rojas et al.~\cite{rojas2012analyzing}.

An alternative approach to dealing with the uncertainty of $G$ is to use system identification techniques to estimate a model $\widehat{G}$ for $G$, and then apply $\Hinf$ algorithms on $\widehat{G}$.  There are two main approaches here: identify the transfer function, or identify a state-space representation.  For transfer function representation, the finite impulse response (FIR) identification procedure proposed by Helmicki et al.~\cite{helmicki91} is the most relevant to the $\Hinf$-estimation problem. An asymptotic analysis of least-squares for FIR identification was carried out by Ljung and Wahlberg~\cite{ljung1992asymptotic}, and more recently finite-time bounds were derived by Tu et al.~\cite{tuboczar} in a probabilistic (non-adversarial) setting. For a comprehensive survey of frequency domain techniques, see Chen and Gu~\cite{chen00}.

For state-space identification, we focus on more recent non-asymptotic results. We place emphasis on non-asymptotic results because they often offer qualitative insight into which techniques are appropriate for a particular problem instance. For more classical asymptotic results, see~\cite{ljung99}. Earlier non-asymptotic results~\cite{campi2002finite, vidyasagar2006learning} featured bounds which were conservative and even exponential in the degree of the system and other quantities. Recent work by Hardt et al.~\cite{hardt2016gradient} shows that only a polynomial number of samples are necessary to recover a state-space representation which generalizes beyond the observed input/output samples. It is not clear, however, how their statistical risk guarantee translates into guarantees for the $\Hinf$-estimation problem.

Finally, we emphasize that the mathematical analysis is primarily driven by the work of B{\"{o}}ttcher and Grudsky, who have published many convergence results of Toeplitz and related operators (see \cite{bottcher1997norms,bottcher2000book}, among others).

\section{Notation}
We quickly fix notation: the closed unit circle in the complex plane is denoted $\Torus \defeq \set{z\in\C}{|z|=1}$ and its interior is denoted $\D$; $\Z_+$ denotes the nonnegative integers including zero and a sequence $u = (u_0,u_1,\dots)$ is said to be in $\ell^2(\Z_+)$ if $\sum_{k=0}^\infty|u_k|^2<\infty$; $L^2$ and $L^\infty$ are the square-integrable and bounded function spaces (equipped with associated norms), respectively; $\norm{\:\cdot\:}$ is the operator norm, whose domain and codomain are made clear in context.

Given a scalar sequence $\{a_k\}_{k \in \Z}$, let $T(a)$ denote the infinite Toeplitz matrix $T(a) = (a_{j-k})_{j,k=0}^{\infty}$, $T_n(a)$ denote the $n \times n$ Toeplitz matrix $T_n(a) = (a_{j-k})_{j,k=0}^{n-1}$ and $H(a) = (a_{j+k+1})_{j,k=0}^{\infty}$ denote the Hankel operator. We equip $L^2(\Torus)$ with the inner product $\ip{f}{g} = \int_{\Torus} f(z) \overline{g(z)} \; dz$, and for consistency we use the convention that $\ip{x}{y} = y^* x$ for two complex vectors in $\C^n$.
\section{Results}

We will work with stable, discrete-time LTI SISO systems only. A stable, discrete-time LTI system $G$ can be described by its impulse response $g \in \ell^2(\Z)$, such that the output $y$ of $G$ with input $u \in \ell^2(\Z)$ is given by the convolution $y = g \ast u$. $G$ can also be described equivalently via its $z$-transform $G(z) = \sum_{k=-\infty}^{\infty} g_k z^{-k}$, where $g = \{g_k\}_{k\in \Z}$. However, for our purposes we will take $G$ to be causal and thus will only consider signals with support on $\Z_+$. A consequence of stability is the fact that $G(z)$ is analytic on the complement of $\mathbb{D}$, and hence belongs to the Hardy space $\Hinf(\Torus)$. This space is equipped with the norm
\begin{align*}
  \norm{G}_\infty := \norm{G}_{\Hinf(\Torus)} = \sup_{z \in \Torus} \abs{G(z)} \:.
\end{align*}
Since the operator $u \mapsto g \ast u$ is linear, it has an infinite-dimensional matrix representation $T(g)$ with respect to the standard basis on $\ell^2(\Z_+)$. It is straightforward to check that $T(g)$ is the lower-triangular Toeplitz matrix $(g_{j-k})_{j,k=0}^{\infty}$, where $g_{-k} = 0$ for $k \geq 1$. A well-known fact dating back to Toeplitz states that $\norm{G}_\infty = \norm{T(g)}$. Furthermore, it is also clear that $\norm{T_n(g)} \leq \norm{G}_\infty$.

The goal of this paper is to compute upper bounds on the quantity
\begin{align*}
  \norm{G}_\infty - \norm{T_n(g)}
\end{align*}
as a function of both $n$ and properties of $G$; namely, in order to be able to choose an $n$ that guarantees $\norm{G}_\infty - \norm{T_n(g)} \leq \epsilon$. By the work of B{\"{o}}ttcher and Grudsky \cite{bottcher2000book}, it is known that this quantity is upper bounded by $O(1/n^2)$, and hence $n \geq \Omega(1/\sqrt{\varepsilon})$ suffices. However, the $O(\cdot)$ of their result hides all parameters of $G$. Our contribution is to provide a system-theoretic upper bound. Specifically, we prove the following statement.
\begin{thm} \label{thm:main_result}
\textbf{(Main Result.)} Let $G(z) = \sum_{k=0}^{\infty} g_k z^{-k}$ be a stable, discrete-time LTI system with stability radius $\rho \in (0, 1)$. Fix a $\gamma \in (\rho, 1)$, and suppose that $\abs{g_0} \leq D$. For all $n \geq 3$, we have that
  \begin{align*}
    \norm{G}_\infty - \norm{T_n(g)} \leq &\; C_1 \frac{D \norm{G^\gamma}_\infty (1-\gamma^2) + \norm{G^\gamma}_\infty^2 \:\gamma}{\norm{G}_\infty(1-\gamma)^4} \frac{1}{n^2} + C_2 \frac{\norm{G^\gamma}^2_\infty}{\norm{G}_\infty (1+\gamma)(1-\gamma)^5} \frac{1}{n^3} \: \notag,
  \end{align*}
where $\norm{G^\gamma}_\infty$ denotes the $\Hinf$-norm of the system $G^\gamma := \gamma G(\gamma z)$, and $C_1, C_2$ are universal constants made explicit in the proof.
\end{thm}

Let us interpret the result of Theorem~\ref{thm:main_result}. First, suppose that $G$ has state-space realization given as $(A, b, c^\T, 0)$. Then, the system $G^\gamma$ has state-space realization $(\gamma^{-1} A, b, c^\T, 0)$. Next, for simplicity, suppose that $D = O(\norm{G^\gamma}_\infty)$. Then, the bound in Theorem~\ref{thm:main_result} can be expressed as
\begin{align*}
  \norm{G}_\infty - \norm{T_n(g)} \leq O\left( \frac{\norm{G^{\gamma}}^2_\infty}{\norm{G}_\infty}\left( \frac{1}{(1-\gamma)^4 n^2} + \frac{1}{(1-\gamma)^5 n^3} \right)   \right) \:.
\end{align*}
For $\varepsilon$ sufficiently small, this bound states that
\begin{align*}
  n \geq \inf_{\rho < \gamma < 1} \Omega\left( \frac{1}{(1-\gamma)^2} \sqrt{\frac{\norm{G^\gamma}^2_\infty}{\norm{G}_\infty} \frac{1}{\varepsilon} } \right)
\end{align*}
is sufficient to ensure that $\norm{G}_\infty - \norm{T_n(g)} \leq \varepsilon$.
\subsection{Proof of Theorem~\ref{thm:main_result}}
We first state an approximation result from
\cite{bottcher2000book}. This is the key result which enables our analysis.
\begin{lem}[Theorem 4.1, \cite{bottcher2000book}]
\label{lemma:approximation}
Suppose that $a \in L^\infty(\Torus)$ such that
$a(z)$ satisfies the following conditions:
  \begin{enumerate}
    \item (Causality) $a_k = 0$ for all $k \leq -1$.
    \item (Decay of Fourier coefficients) There exists $C > 0$ and $\rho \in (0, 1)$ such that
      for all $k \geq 1$, we have $\abs{a_k} \leq C \rho^{k-1}$.
    \item (Smoothness) Let $\theta_0 \in [-\pi, +\pi]$ satisfy $a(e^{j\theta_0}) = \norm{a}_\infty$.
      There exists $L > 0$ such that for all $\theta \in [-\pi, +\pi]$,
      \begin{align*}
        \norm{a}^2_\infty - \abs{a(e^{j\theta})}^2 \leq L \abs{\theta - \theta_0}^2 \:.
      \end{align*}
  \end{enumerate}
Assume that $n \geq 3$. Then, there exists universal constants $C_1, C_2 > 0$ such that
  \begin{align*}
    \norm{T_n(a)} \geq &\; \norm{a}_{\infty} - C_1 \frac{L}{\norm{a}_\infty} \frac{1}{n^2} - C_2 \frac{C^2}{\norm{a}_\infty(1+\rho)(1-\rho)^5} \frac{1}{n^3} \:.
  \end{align*}
\end{lem}
\begin{proof}
We modify the proof given in Theorem 4.1 of \cite{bottcher2000book},
using the stated assumptions to make the necessary simplifications along the way.
For completeness, we include the full presentation in the appendix.
\end{proof}

The remainder of the proof of Theorem~\ref{thm:main_result} involves satisfying the hypothesis of Lemma~\ref{lemma:approximation}.  Indeed, our next step is to control the smoothness constant $L$ in Lemma~\ref{lemma:approximation}.
\begin{lem}
\label{lemma:smoothness}
Let a function $a \in L^\infty(\Torus)$ be given by $a(z) = \sum_{k=0}^{\infty} a_k z^k$ and suppose for all $k \geq 1$ we have $\abs{a_k} \leq C \rho^{k-1}$ with $\rho \in (0, 1)$, and $\abs{a_0} \leq D$. Then, we have that for all $\theta$,
\begin{align*}
  \norm{a}^2_\infty - \abs{a(e^{j\theta})}^2 \leq \frac{DC(1-\rho^2) + C^2\rho}{(1-\rho)^4} (\theta - \theta_0)^2 \:.
\end{align*}
\end{lem}
\begin{proof}
Fix an integer $N \geq 1$ and define $a_N(z) := \sum_{k=0}^{N-1} a_k z^k$. Let $\theta_0$ denote a frequency such that $\abs{a_N(e^{j\theta_0})} = \norm{a_N}_\infty$. Denote $M_N := \norm{a_N}_\infty$, and define the function $g_N(\theta) := M_N^2 - \abs{a_N(e^{j\theta})}^2$. It is easy to see that $g_N(\theta) \geq 0$ and $g_N(\theta)$ is infinitely differentiable. Note that $g_N$ achieves a global minimum of zero at $\theta_0$, and hence $g_N(\theta_0) = g_N'(\theta_0) = 0$. Using Taylor's theorem, there exists some point $c \in (\theta_0, \theta)$ such that
\begin{align}
  g_N(\theta) &= g_N(\theta_0) + g_N'(\theta_0)(\theta - \theta_0) + \frac{1}{2} g_N''(c)(\theta- \theta_0)^2 \nonumber \\
    &\leq \sup_{\theta \in [-\pi, +\pi]} \abs{g_N''(\theta)} \cdot  \frac{1}{2} (\theta-\theta_0)^2 \label{eq:taylor_estimate} \:.
\end{align}
Expanding $\abs{a_N(e^{j\theta})}^2$ gives
\begin{align*}
  \abs{a_N(e^{j\theta})}^2 &= a_N(e^{j\theta}) \overline{a_N(e^{j\theta})} \\
    &= \left( \sum_{k=0}^{N-1} a_k e^{j k\theta} \right) \left( \sum_{k=0}^{N-1} \overline{a_k} e^{-j k \theta} \right) \\
    &= \sum_{k,l=0}^{N-1} a_k \overline{a_l} e^{j\theta(k-l)} \\
    & = \sum_{\Delta=-(N-1)}^{N-1} b_{\Delta} e^{j\theta\Delta} = \sum_{\Delta=-(N-1)}^{N-1} b_{\Delta} \cos(\theta\Delta) \:,
\end{align*}
where $b_\Delta := \sum_{k=\max(0, \Delta)}^{\min(N-1, N-1+\Delta)} a_{k} \overline{a_{k-\Delta}}$. Note that $b_{\Delta} = \overline{b_{-\Delta}}$. Using this calculation,
\begin{align*}
    \frac{d^2}{d\theta^2} \abs{a_N(e^{j\theta})}^2 = - \sum_{\Delta=-(N-1)}^{N-1} \Delta^2 b_{\Delta} \cos(\theta \Delta) \:.
\end{align*}
Hence,
\begin{align*}
  \abs{g_N''(\theta)} \leq 2 \sum_{k = 0}^{N-1} k^2 \abs{b_k} \leq 2 \sum_{k = 0}^{\infty} k^2 \abs{b_k} \:.
\end{align*}
Using the decay assumption, for $m \geq 1$,
\begin{align*}
  \abs{b_m} \leq \sum_{l=0}^{\infty} \abs{a_l}\abs{a_{l+m}} \leq &\; DC \rho^{m-1} + \sum_{l=1}^{\infty} \abs{a_l}\abs{a_{l+m}} \\
  \leq &\;  DC \rho^{m-1} + C^2\sum_{l=0}^{\infty} \rho^{l} \rho^{l+m} \\
  = &\;  DC\rho^{m-1} + \frac{C^2 \rho^m }{1-\rho^2} \:.
\end{align*}
Therefore,
\begin{align*}
  \abs{g_N''(\theta)} \leq &\;  2  \sum_{k=1}^{\infty} k^2 \abs{b_k} = 2DC \sum_{k=1}^{\infty} k^2 \rho^{k-1} + \frac{2C^2}{1-\rho^2} \sum_{k=1}^{\infty} k^2 \rho^k \:.
\end{align*}
Now,
  \begin{align*}
    \sum_{k=1}^{\infty} k^2 \rho^{k-1} & = 1 + \sum_{k=1}^{\infty} (k+1)^2 \rho^k = \frac{1+\rho}{(1-\rho)^3}\:,\; \text{and} \\
    \sum_{k=1}^{\infty} k^2 \rho^k &= \frac{\rho(1+\rho)}{(1-\rho)^3} \:.
  \end{align*}
Hence,
  \begin{align*}
    \abs{g_N''(\theta)} \leq 2DC \frac{1+\rho}{(1-\rho)^3} + 2C^2 \frac{\rho}{(1-\rho)^4} \:.
  \end{align*}
Plugging this estimate into \eqref{eq:taylor_estimate}, we have for all $\theta$,
    \begin{align}
      M_N^2 - & \abs{a_N(e^{j\theta})}^2 \leq \left(DC \frac{1+\rho}{(1-\rho)^3} + C^2 \frac{\rho}{(1-\rho)^4}\right)  (\theta - \theta_0)^2 \:. \label{eq:decay_truncation}
    \end{align}
The above presentation was valid for arbitrary $N \geq 1$. Define $g(\theta) := \norm{a}^2_\infty - \abs{a(e^{j\theta})}^2$. The claim now follows by passing to the limit in \eqref{eq:decay_truncation}.
\end{proof}
Lemma~\ref{lemma:smoothness} cannot be improved in general. Consider the single pole system $G(z) = \frac{C}{z-\rho} + D$ with $\rho \in (0, 1)$, $C > 0$ and $D > 0$. The impulse response coefficients of $G(z)$ are all positive, and hence the $\Hinf$-norm is achieved at $\theta = 0$. A simple calculation yields that
\begin{align*}
  \frac{d^2}{d\theta^2} \abs{G(e^{j\theta})}^2 \biggl|_{\theta=0} = -\frac{2DC(1-\rho^2) + 2C^2\rho}{(1-\rho)^4} \:.
\end{align*}
Hence, for $\theta$ small, we have that
\begin{align*}
  \norm{G}^2_{\infty} - \abs{G(e^{j\theta})}^2 = \frac{DC(1-\rho^2) + C^2\rho}{(1-\rho)^4} + O(\theta^3) \:.
\end{align*}
This is the same behavior predicted by Lemma~\ref{lemma:smoothness}.

Finally, we bound the $C$ constant in the decay assumption of Lemma~\ref{lemma:approximation} in terms of the $\Hinf$-norm of $G$. The following estimate follows from Cauchy's integral formula.
\begin{lem}[Lemma 1, \cite{goldenshluger01}]
\label{lemma:decay}
Let $G(z) = \sum_{k=0}^{\infty} a_k z^{-k}$ be a stable SISO LTI system with stability radius $\rho \in (0, 1)$. Fix any $\gamma \in (\rho, 1)$. Then, for all $k \geq 1$,
  \begin{align*}
    \abs{a_k} \leq &\; \norm{G(\gamma z)}_\infty \gamma^k \:.
  \end{align*}
\end{lem}
At this point, Theorem~\ref{thm:main_result} follows directly by combining Lemma~\ref{lemma:approximation}, Lemma~\ref{lemma:smoothness}, and Lemma~\ref{lemma:decay}.

\subsection{Optimality of Theorem~\ref{thm:main_result}}
\label{sec:results:necessity}
The $O(1/n^2)$ dependence of Theorem~\ref{thm:main_result} cannot be further improved in general. We note that a similar calculation appears in Section 4.4 of \cite{bottcher2000book} regarding the convergence of norms of inverses of Toeplitz matrices.

Fix $a_0, a_1 \in \R$ and suppose both $a_0 \geq 0$ and $a_1 \geq 0$. Form the symbol $a(z) = a_0 + a_1 z$. A quick calculation shows that $\norm{a}_\infty = a_0 + a_1$. This is a special case of the fact that, if a sequence $\{a_k\}$ is non-negative, then $\norm{a}_\infty = \sum_{k} a_k$. Now, let us upper bound $\norm{T_n(a)}$. First, we observe that $T_n(a)^* T_n(a)$ has special structure,
\begin{align*}
  T_n(a)^* T_n(a) = &\; \begin{bmatrix}
    a_0^2 + a_1^2 & a_0 a_1 & 0 & \cdots & 0 \\
    a_0 a_1 & a_0^2 + a_1^2 & \ddots & \ddots & \vdots \\
    0 & \ddots & \ddots & \ddots &  0\\
    \vdots & \ddots & \ddots & a_0^2 + a_1^2 & a_0 a_1 \\
    0 & \cdots & 0 & a_0 a_1 & a_0^2
  \end{bmatrix} \\
  = &\; K - a_1^2 e_ne_n^* \:,
\end{align*}
where $K$ is a symmetric tridiagonal matrix with diagonal entries $a_0^2 + a_1^2$ and off-diagonal entries $a_0 a_1$. From the formula for the eigenvalues of a tridiagonal matrix, we have that
\begin{align*}
  \lambda_{\max}(K) = a_0^2 + a_1^2 + 2a_0a_1 \cos(\pi/(n+1)) \:.
\end{align*}
Furthermore, since $K - a_1^2 e_ne_n^* \preccurlyeq K$, we have by the Courant minimax principle that 
\begin{align*}
\lambda_{\max}(K - a_1^2 e_ne_n^*) \leq \lambda_{\max}(K)\:,
\end{align*} 
and hence,
\begin{align*}
  \norm{T_n(a)}^2 &= \lambda_{\max}(T_n(a)^* T_n(a)) \leq \lambda_{\max}(K) \\
  &= a_0^2 + a_1^2 + 2a_0a_1 \cos(\pi/(n+1)) \\
  &= \norm{T(a)}^2 - 2a_0a_1(1 - \cos(\pi/(n+1))) \:.
\end{align*}
Now, using the fact that $1-\cos(x) \geq \frac{8}{\pi^2}\left(2-\sqrt{2}\right) x^2$ for all $x \in [0, \pi/4]$, we conclude for all $n \geq 3$,
\begin{align*}
  \norm{T_n(a)}^2 \leq &\;\norm{T(a)}^2 - \frac{16}{\pi^2}\left(2-\sqrt{2}\right) \pi^2 a_0a_1/(n+1)^2 \\
  = &\; \norm{T(a)}^2 - O(a_0 a_1 n^{-2})\:.
\end{align*}
Finally, using the fact that the square root is concave, we have for all positive reals $s, t$ that $\sqrt{t} \leq \sqrt{s} + \frac{1}{2\sqrt{s}} (t - s)$. Hence, for $n$ sufficiently large,
\begin{align*}
  \norm{T_n(a)} &\leq \sqrt{\norm{T(a)}^2 - O(a_0a_1 n^{-2})} \\
                &\leq \norm{T(a)} - O\left( \frac{a_0a_1}{\norm{T(a)}} n^{-2}\right) \:.
\end{align*}
Thus, a general estimate of $\norm{T_n(a)} \geq \norm{T(a)} - O(n^{-2})$ is not improvable without further assumptions. Furthermore, if $a_0=a_1=a$, then $a_0a_1 = a^2 = \norm{T(a)}^2_\infty/4$, in which case the bound from above simplifies to
\begin{align*}
  \norm{T_n(a)} \leq \norm{T(a)} - O(\norm{T(a)} n^{-2}) \:.
\end{align*}

\section{Comparison to FIR Identification}
The filter $G(z) = a + a z^{-1}$ showcases an interesting gap between Wahlberg et al.'s method and the FIR system identification method proposed by Helmicki et al.~\cite{helmicki91} and analyzed in a probabilistic setting by Tu et al.~\cite{tuboczar}. Specifically, in the setting of Tu et al., one has access to $G$ via noisy measurements
\begin{align*}
    Y_{u} = (g \ast u)_{k=0}^{1} + \xi \:, \:\: \xi \sim N(0, \sigma^2 I) \:,
\end{align*}
where $u$ is restricted to satisfy $\norm{u}_2 \leq 1$. Theorem 1.1 from \cite{tuboczar} asserts that with probability at least $1-\delta$ over the randomness of the $\xi$'s, one can identify a length 2 FIR filter $\widetilde{G}(z) = \widetilde{g_0} + \widetilde{g_1} z^{-1}$ satisfying $\norm{G - \widetilde{G}}_\infty \leq \varepsilon$ with at most $T = O( \frac{\sigma^2}{\varepsilon^2}  \log(1/\delta) )$ timesteps. Observe that this bound is independent of the magnitude of $a$. Of course, by triangle inequality, the guarantee $\norm{G-\widetilde{G}}_\infty \leq \varepsilon$ implies the guarantee $\abs{\norm{G}_\infty - \norm{\widetilde{G}}_\infty} \leq \varepsilon$.

On the other hand, the calculations in Section~\ref{sec:results:necessity} show that $\norm{G}_\infty - \norm{T_n(G)} \geq \Omega(\norm{G}_\infty n^{-2} ) $.  Hence, for a fixed $n$, as $a \rightarrow \infty$, the gap $\norm{G}_\infty - \norm{T_n(G)}$ grows arbitrarily large as well.  That is, the length $n$ needed to ensure that $\norm{G} - \norm{T_n(G)} \leq \varepsilon$ is arbitrarily large.
\section{Conclusion}

In this paper, we provided a non-asymptotic bound on the convergence rate of $\norm{T_n(g)}$ to $\norm{T(g)}$, utilizing the work of B{\"{o}}ttcher and Grudsky. We note that our bounds are only the first step in providing a finite-time rate of convergence of Wahlberg et al.'s method. Unfortunately, existing analysis of the (noisy) power method (see e.g. \cite{hardt14}) requires control of the eigenvalue gap $\lambda_1/\lambda_2$. In this setting, we have given upper bounds on $\lambda_1$. A lower bound on $\lambda_2$, however, involves lower bounding the second singular value of $T_n(g)$, a much more non-trivial task. We leave this analysis to future work.

Finally, we leave open the question of whether or not, from an information-theoretic standpoint, there is a gap between the sample complexity of $\Hinf$-norm estimation versus FIR identification.

\section*{Acknowledgements}
The authors thank Andrew Packard for helpful discussions and valuable feedback. RB is supported by the Department of Defense NDSEG Scholarship. BR is generously supported by NSF award CCF-1359814, ONR awards N00014-14-1-0024 and N00014-17-1-2191, the DARPA Fundamental Limits of Learning (Fun LoL) Program, a Sloan Research Fellowship, and a Google Faculty Award.

\bibliographystyle{abbrv}
\bibliography{IEEEabrv,paper}

\clearpage

\appendix
\renewcommand{\thesection}{A}
\phantomsection
\section*{Appendix: Proof of Lemma~\ref{lemma:approximation}}

We first summarize the notation and basic techniques from Fourier analysis used in \cite{bottcher2000book}.

The Fourier transform $\F : L^2(\Torus) \longrightarrow \ell^2(\Z)$ is given by the convention
\begin{align*}
  a_n := (\F f)(n) = \frac{1}{2\pi} \int_{\Torus} f(z) \overline{\phi_n(z)} \; dz \:, \:\: \phi_n(z) = z^n \:.
\end{align*}
Hence we have the Fourier series $f = \sum_{n \in \Z} a_n \phi_n$ for $f \in L^2(\Torus)$. Given $a \in L^2(\Torus)$, let $\widetilde{a}(z) := a(z^{-1})$. Furthermore, let $P_n$ and $W_n$ denote (resp.) the projection and reversed projection operators
\begin{align*}
  & P_n :  \ell^2(\Z_+) \rightarrow \ell^2(\Z_+) \::= (x_0, x_1, x_2, \ldots) \mapsto (x_0, x_1, \ldots, x_{n-1}, 0, \ldots) \:, \\
  & W_n :  \ell^2(\Z_+) \rightarrow  \ell^2(\Z_+) \::= (x_0, x_1, x_2, \ldots) \mapsto (x_{n-1}, x_{n-2}, \ldots, x_{0}, 0, \ldots) \:.
\end{align*}
Finally, let us define the trigonometric polynomials $p_m^t$ for positive integers $m$ and $t$ as
\begin{align*}
  p_m^t(z) = (1 + z + \ldots + z^m)^t \:.
\end{align*}
The following propositions from B\"ottcher and Grudsky \cite{bottcher2000book} all follow from direct manipulations.
\begin{prop}
Given $a \in L^2(\Torus)$ and $x,y \in \C^n$, we have
\begin{align}
  & \ip{T_n(a) x}{y} = \frac{1}{2\pi} \int_{\Torus} a(z) f(z) \overline{g(z)} \; dz \:,\label{eq:toeplitz_quad_form}\\
  & \text{where } f = \sum_{k=0}^{n-1} x_k \phi_k \:, \:\: g = \sum_{k=0}^{n-1} y_k \phi_k \:.
\end{align}
\end{prop}

\begin{prop}[Widom \cite{widom1976asymptotic}]
\label{prop:decomp}
Given $a, b \in L^2(\Torus)$, we have
\begin{align*}
  T_n(a) T_n(b) = &\; T_n(ab) - P_n H(a) H(\widetilde{b}) P_n \\
  &\; {-} \: W_n H(\widetilde{a}) H(b) W_n \:.
\end{align*}
\end{prop}

\begin{prop}
Given positive integers $m$ and $t$, we have
\begin{align}
  p_m^t(e^{j\theta}) = e^{jmt \theta/2} \left(\frac{\sin((m+1)\theta/2)}{\sin(\theta/2)} \right)^{t} \:. \label{eq:trig_poly_series}
\end{align}
\end{prop}
The next lemma estimates the growth of the $L^2(\Torus)$ norm of $p_m^t$ from below, as a function of $m$ and $t$.
\begin{lem}[Lemma 4.2, \cite{bottcher1998toeplitz}]
\label{lemma:trig_lower_bound}
Given positive integers $m$ and $t$, we have
  \begin{align*}
    \norm{p^t_m}_2^2 \geq \frac{16}{9\pi} \frac{1}{\sqrt{t}} (m+1)^{2t-1} \:. 
  \end{align*}
\end{lem}
We first prove the case when $\theta_0 = 0$. We will then argue that this is without loss of generality. First, observe that by the causality assumption, both $H(\widetilde{a})$ and $H(\overline{a})$ are the zero operator, and hence Proposition~\ref{prop:decomp} yields\footnote{Recall that $a$ is a function with Fourier coefficients $a_k$, so the $k$th Fourier coefficient of the function $\overline{a}$ is the conjugate of the $-k$th coefficient of $a$.} 
\begin{align}
  T_n(a) T_n(a)^* &\; = T_n(a) T_n(\overline{a}) = T_n(\abs{a}^2) - P_n H(a) H(\widetilde{\overline{a}}) P_n \:. \label{eq:toeplitz_decomp}
\end{align}

Put $m = \lceil \frac{n}{2} - 1 \rceil$. One readily checks that $2m < n \leq 2(m+1)$. Now, define $M := \norm{a}_\infty$. Observe that the matrix $M^2 I_n = T_n(M^2)$, where on the right hand side we treat $M^2$ as a constant function. Now, define $K := H(a) H(\widetilde{\overline{a}})$. For any $x \in \C^n$, multiplying both sides of \eqref{eq:toeplitz_decomp} by $x^*$ and $x$ respectively gives
\begin{align}
    \norm{T_n(a)}^2  \norm{x}^2_2 \geq \ip{T_n(a) T_n(\overline{a}) x}{x}
    = &\; \ip{T_n(\abs{a}^2) x}{x} - \ip{P_n K P_n x}{x} \notag \\
    = &\; M^2 \norm{x}^2_2 + \ip{(T_n(\abs{a}^2) - T_n(M^2))x}{x} - \ip{P_n K P_n x}{x} \notag \\
    \geq &\; M^2 \norm{x}^2_2 - \abs{\ip{T_n(M^2 - \abs{a}^2)x}{x}} - \ip{P_n K P_n x}{x} \:. \label{eq:starting_point}
  \end{align}
The remainder of the proof consists of choosing a particular $x$ such that bounding the 2nd and 3rd term in the last line above yields the desired inequality. The trick is to use the Fourier coefficients of $p_m^2$. By Lemma~\ref{lemma:trig_lower_bound} and the inequality $2(m+1) \geq n$,
\begin{align}
  \norm{p_m^2}_2^2 \geq \frac{16}{9 \sqrt{2}\pi} (m+1)^{3} \geq \frac{2}{9\sqrt{2} \pi} n^3 \:. \label{eq:polynomial_lower_bound}
\end{align}
Now set $x= (x_0, x_1, \ldots, x_{n-1})$ such that $x_k$ is the $k$-th Fourier coefficient of the polynomial $p_m^2$. Explicitly,
\begin{align}
x_k = \begin{cases}
    k + 1 &\text{if $0 \leq k \leq m$} \\
    2m - k + 1 &\text{if $m < k \leq 2m$}
\end{cases} \:. \label{eq:x_coeff}
\end{align}

Since $n > 2m$, $x$ includes all the non-zero Fourier coefficients of $p_m^2$.
By Parseval's identity, we have that $\norm{x}_2^2 = \frac{1}{2\pi} \norm{p_m^2}_2^2$, and hence
in light of \eqref{eq:polynomial_lower_bound},
\begin{align*}
  \norm{x}_2^2 \geq \frac{1}{9\sqrt{2} \pi^2} n^3 \:.
\end{align*}
Now combining \eqref{eq:toeplitz_quad_form} with the identity \eqref{eq:trig_poly_series},
\begin{align*}
  \abs{\ip{T_n(M^2 - \abs{a}^2) x}{x}}
    &\;=\frac{1}{2\pi} \int_{-\pi}^{\pi} (M^2 - \abs{a(e^{j\theta})}^2) \left( \frac{\sin((m+1)\theta/2)}{\sin(\theta/2)} \right)^{4} \; d\theta \\
    &\;\stackrel{(a)}{\leq} \frac{L}{2\pi} \int_{-\pi}^{\pi} \abs{\theta}^{2}\left( \frac{\sin((m+1)\theta/2)}{\sin(\theta/2)} \right)^{4} \; d\theta \\
    &\;=\frac{L}{2\pi} \int_{0 < \abs{\theta} < 1/(m+1)}\abs{\theta}^{2}\left( \frac{\sin((m+1)\theta/2)}{\sin(\theta/2)} \right)^{4} \; d\theta \\
    &\qquad{+}\:\frac{L}{2\pi} \int_{1/(m+1) < \abs{\theta} < \pi}\abs{\theta}^{2}\left( \frac{\sin((m+1)\theta/2)}{\sin(\theta/2)} \right)^{4} \; d\theta \\
    &\;\stackrel{(b)}{\leq}\frac{L}{2\pi} \int_{0 < \abs{\theta} < 1/(m+1)}\abs{\theta}^{2}  (m+1)^{4} \; d\theta \\
    &\qquad{+}\:\frac{L}{2\pi} \int_{1/(m+1) < \abs{\theta} < \pi}\abs{\theta}^{2}\left( \frac{\pi}{\abs{\theta}} \right)^{4} \; d\theta \\
    &\;\leq \frac{2L}{3\pi}(m+1) + L\pi^3(m+1) \\
    &\;= \frac{2+3\pi^4}{3\pi} L (m+1) \leq \frac{2+3\pi^4}{3\pi} L n \:.
\end{align*}
Above, (a) follows from the smoothness assumption and (b) follows by using the inequalities (i) for a positive integer $t$, $0 \leq \sin(t \theta)/\sin(\theta) \leq t$ for $\abs{\theta} \leq 1/2t$ and (ii) $\abs{\sin(\theta)} \geq \frac{2}{\pi} \abs{\theta}$ for all $\theta \in [-\pi/2, \pi/2]$. This yields an estimate of the first term.

We now proceed to estimate the second term. By definition of $K$,
\begin{align*}
  K = \left( \sum_{k=0}^{\infty} a_{i+1+k} \overline{a_{j+1+k}} \right)_{i,j=0}^{\infty} \:.
\end{align*}
Using our decay assumption $\abs{a_k} \leq C \rho^{k-1}$ for every $k \geq 1$,
\begin{align*}
  \abs{K_{ij}} \leq C^2 \sum_{k=0}^{\infty} \rho^{i+k} \rho^{j+k} =  C^2 \rho^{i+j} \sum_{k=0}^{\infty} \rho^{2k} = \frac{C^2}{1-\rho^2} \rho^{i+j} \:.
\end{align*}
Therefore,
  \begin{align*}
    \ip{P_n K P_n x}{x}\leq \sum_{i,j=0}^{n-1} \abs{K_{ij}} \abs{x_i} \abs{x_j} \leq \frac{C^2}{1-\rho^2} \sum_{i,j=0}^{n-1} \rho^{i+j} \abs{x_i}\abs{x_j} = &\; \frac{C^2}{1-\rho^2} \left(\sum_{i=0}^{n-1} \rho^i \abs{x_i} \right)^2 \\
    =  &\;\frac{C^2}{1-\rho^2} \left(\sum_{i=0}^{2m} \rho^i \abs{x_i} \right)^2 \:.
  \end{align*}
Next, we observe that
\begin{align*}
 \sum_{i=0}^{2m} \rho^i \abs{x_i} = (1 + \rho + \rho^2 + \ldots + \rho^{m})^2 \leq \frac{1}{(1-\rho)^2} \:.
\end{align*}
The first equality holds due to our particular choice of coefficients $x_k$
from \eqref{eq:x_coeff}. Combining our calculations,
\begin{align*}
  \ip{P_n K P_n x}{x} \leq \frac{C^2}{(1+\rho)(1-\rho)^5} \:.
\end{align*}
This yields a bound on the 3rd term in \eqref{eq:starting_point}. Collecting these bounds,
\begin{align*}
  \norm{T_n(a)}^2 \geq &\; M^2 -  \frac{2 + 3\pi^4}{3\pi} \frac{Ln}{\norm{x}^2_2} - \frac{1}{\norm{x}^2_2}\frac{C^2}{(1+\rho)(1-\rho)^5} \\
  \geq &\;M^2 - 3\sqrt{2} \pi(2+3\pi^4)\frac{L}{n^2} - 9\sqrt{2}\pi^2\frac{C^2}{(1+\rho)(1-\rho)^5} \frac{1}{n^3} \:.
\end{align*}
The above inequality tells us that $M^2 - \norm{T_n(a)}^2 \leq Q(n)$ for some $Q(n)$ positive.
Since 
\begin{align*}
M^2 - \norm{T_n(a)}^2 = (M + \norm{T_n(a)})(M - \norm{T_n(a)}) 
\geq M (M - \norm{T_n(a)})\:,
\end{align*}
we conclude that $M - \norm{T_n(a)} \leq Q(n)/M$. Rearranging, this finally gives
\begin{align*}
  \norm{T_n(a)} \geq  M - 3\sqrt{2}\pi(2+3\pi^4) \frac{L}{M} \frac{1}{n^2} -\: 9\sqrt{2}\pi^2 \frac{C^2}{M(1+\rho)(1-\rho)^5} \frac{1}{n^3} \:.
\end{align*}

It now remains to argue that $\theta_0 = 0$ without loss of generality. Define $g(e^{j\theta}) = a(e^{j(\theta_0 + \theta)})$. Then $\norm{g}_\infty = M$ and $\abs{g(e^{j 0})} = \abs{a(e^{j\theta_0})} = M$. It is easy to check that $g$ satisfies all the assumptions, so we know that the result holds for $g$. Furthermore, we also have that $\norm{T_n(g)} = \norm{T_n(a)}$ and $\norm{T(g)} = \norm{T(a)}$, and hence the result holds for $a$ as well.

\end{document}